\pdfoutput=1
\documentclass[11pt,a4paper]{amsart}
\usepackage[margin=16mm]{geometry}
\usepackage{verbatim}
\usepackage{amssymb,amsmath,amsthm}

\newtheorem{theorem}{Theorem}[section]

\newtheorem{corollary}{Corollary}[section]

\newtheorem{lemma}{Lemma}[section]

\newtheorem{proposition}{Proposition}[section]

\theoremstyle{definition}
\newtheorem{definition}{Definition}[section]
\newtheorem{remark}{Remark}[section]
\newtheorem{warning}{Warning}[section]

\newcommand{\x}{\boldsymbol{x}}
\newcommand{\pp}{\boldsymbol{p}}

\newcommand{\R}{\mathbb{R}}

\newcommand{\p}{\mathbb{P}}
\newcommand{\EE}{\mathbb{E}}

\newcommand{\SSS}{\mathbb{S}}
\newcommand{\g}{\mathfrak g}
\newcommand{\gh}{\mathfrak h}

\newcommand{\so}{\mathfrak{so}}

\DeclareMathOperator{\tr}{trace}

\DeclareMathOperator{\id}{id}

\DeclareMathOperator{\SO}{\mathsf{SO}}
\DeclareMathOperator{\OOO}{\mathsf{O}}
\DeclareMathOperator{\CO}{\mathsf{CO}}

\DeclareMathOperator{\GL}{\mathsf{GL}}

\DeclareMathOperator{\SE}{\mathsf{SE}}

\DeclareMathOperator{\Gr}{Gr}

\DeclareMathOperator{\Aff}{Aff}

\newcommand{\E}{\mathcal{E}}

\newcommand{\CC}{\mathcal{C}}

\newcommand{\TT}{\mathcal{T}}

\newcommand{\NN}{\mathcal{N}}

\newcommand{\Th}{^\textrm{th}}

\begin{document}

\title{A general method to construct invariant PDEs on homogeneous manifolds}
   \author{Dmitri V. Alekseevsky}
  \address{Institute for Information Transmission Problems, B. Karetny
per. 19, 127051, Moscow (Russia) and University of Hradec Kr\'alov\'e,
Faculty of Science, Rokitansk\'eho 62, 500~03 Hradec Kr\'alov\'e,  Czech Republic}
\email{dalekseevsky@iitp.ru}
\author{Jan Gutt}
\address{Center for Theoretical Physics of the Polish Academy of Sciences,
Al. Lotnik\'ow 32/46,
02-668 Warsaw,
Poland}
 \email{jan.gutt@gmail.com}
 \author{Gianni Manno}
   \address{Dipartimento di Matematica ``G. L. Lagrange'', Politecnico di Torino, Corso Duca degli Abruzzi, 24, 10129 Torino, Italy.}
    \email{giovanni.manno@polito.it}
 \author{Giovanni Moreno}
 \address{Department of Mathematical Methods in Physics,
 Faculty of Physics, University of Warsaw,
ul. Pasteura 5, 02-093 Warszawa, Poland}
 \email{giovanni.moreno@fuw.edu.pl}
\date{\today}
\maketitle

\begin{abstract}
Let  $M = G/H$  be  an  $(n+1)$-dimensional  homogeneous manifold  and  $J^k(n,M)=:J^k$ be  the manifold of $k$-jets  of hypersurfaces of $M$. The Lie group $G$ acts naturally on each  $J^k$.  A $G$--invariant partial differential equation   of order $k$ for hypersurfaces of $M$ (i.e., with $n$ independent variables and $1$ dependent one) is defined  as a  $G$--invariant hypersurface  $\E \subset  J^k$.
We describe  a general method   for   constructing   such  invariant  partial differential equations  for $k\geq 2$.   The problem  reduces  to  the description  of  hypersurfaces,  in  a certain  vector  space,  which are invariant  with  respect  to  the linear action  of  the  stability  subgroup $H^{(k-1)}$ of the $(k-1)$--prolonged action of $G$. We  apply  this  approach  to  describe  invariant  partial differential equations  for hypersurfaces  in the Euclidean   space $\EE^{n+1 }$
and in the conformal  space $\SSS^{n+1}$.
%
Our method works under some mild assumptions on the action of $G$, namely:  A1) the  group $G$ must have  an open   orbit  in $J^{k-1}$, and A2) the stabilizer $H^{(k-1)}\subset G$   of the fibre  $J^k\to J^{k-1}$ must factorize via the group of translations  of the fibre itself.


\end{abstract}


\smallskip

\textbf{MSC Classification 2010:} 58A20, 58J70, 53C30, 35A30.

\setcounter{tocdepth}{2}

%

\section{Introduction}

\subsection{Starting point}

In this paper we continue a research program started by us in \cite{MR3904635} and, from a slightly different point of view, by   D. The in    \cite{The2018}, that is  developing tailor--made geometric and algebraic methods to explicitly construct  partial differential equations (PDEs, for short) that  admit a given group of symmetries. The problem itself is rather old and classical: its origins date   back to the works of Lie, Darboux, Cartan and others. In what follows,
we  consider  the global problem of  describing  $G$--invariant  PDEs, where $G$ is a Lie  group acting  transitively on an $(n+1)$--dimensional homogeneous manifold $J^0=M =G/H$: such manifold is interpreted as the space of $n$ independent variables and a dependent one. We propose a method to construct such invariant PDEs: our  approach  admits also a local  reformulation  in terms  of the Lie  algebra  $\mathfrak{g}$ of the group $G$ and, as such, it can be regarded as a  method  to construct  $\mathfrak{g}$--invariant PDEs. One of the key tool of our analysis is going to be the affine structure of the bundles $\pi_{\ell,\ell-1}:J^{\ell}(n,M)\longrightarrow J^{\ell-1}(n,M)$
of the spaces of $\ell$--jets of $n$--dimensional embedded submanifolds of $M$ (that is, hypersurfaces of $M$),  for  $\ell\geq 2$.
By contrast, in the aforementioned paper \cite{MR3904635}, the authors started from a homogeneous  $(2n+1)$--dimensional \emph{contact} manifold for a complex simple Lie group and looked for invariant hypersurfaces, in the corresponding Lagrangian Grassmanian, whose algebraic degree (measured via the Pl\"ucker embedding) is minimal: only in some special cases such a minimal degree is attained by the so--called Lagrangian Chow transform of the sub--adjoint varieties, which, in general, display a very high degree;  the same output of the Lagrangian Chow transform   can be obtained by the original techniques, based on Jordan algebras instead, developed by    D. The in \cite{The2018}.\par
%
In the classical language of symmetries of PDEs, when a group $G$ acts on $M$ we say that it acts by \emph{point transformations}, whereas when $G$ acts on the contact manifold $J^1(n,M)$ we speak of \emph{contact transformations} instead (see, e.g., \cite{MR1670044}); it is well known that not all contact manifolds are the projectivized cotangent bundle of a manifold and, even when they are, not all contact transformations can be obtained by lifting point ones. More precisely, in \cite{MR3904635} the departing point is a homogeneous contact manifold  with respect to a \emph{complex}  simple Lie group of  contact transformations, whereas in  the present work we deal with \emph{real} Lie groups acting on $M=J^0(n,M)$  and satisfying some mild assumptions, see Section \ref{sec:descriptions} below.

\subsection{Preliminary definitions}\label{sec:preliminaries}

Throughout this paper  $M=G/H$ will be an $(n+1)$--dimensional homogeneous manifold
and  $S \subset M$ an embedded  hypersurface of $M$, unless  otherwise specified. Locally,  in an appropriate  local chart
\begin{equation}\label{eqn:coord.u.x}
(u,\x)=(u,x^1, \ldots, x^n)
\end{equation}
of $M$, the hypersurface $S$  can  be  described  by an  equation   $u = f(\x) =  f(x^1, \dots, x^n)$, where  $f$~is  a smooth function    of the variables $x^1, \dots, x^n$, that we refer to as the \emph{independent} variables, to   distinguish them  from the remaining   coordinate $u$,  that is the \emph{dependent} one.\footnote{A reader 
who is familiar with the standard literature about jet spaces may have noticed that we reversed the order of $\x$ and $u$: this choice will be more convenient for us as the coordinate $u$ will play the role of the ``$0^{\textrm{th}}$ coordinate''.
}
 We say that such  a chart is  \emph{admissible} for $S$ or, equivalently, that the hypersurface $S$ is (locally) admissible for the chart $(u,\x)$.
We denote by $S_f=S$ the graph of $f$:
\begin{equation*}
S_f:=\{\big(f(\x)\,,\x\big)\}=\{u=f(\x)\}\, .
\end{equation*}
%

Given two hypersurfaces $S_1$ and $S_2$ through a common point $\pp $, one can always choose a chart $(u,\x)$ near $\pp$ that is admissible for both: $S_1=S_{f_1}$, $S_2=S_{f_2}$. This paves the ground for  the  following definition: even if it is   intrinsically geometric, we  rather give it in a coordinate--wise form to better fit the general approach of the paper;   standard techniques allow to show its independence of the choice of coordinates.

\begin{definition}\label{defDmitriGetti}
Two hypersurfaces $S_{f_1}, \, S_{f_2}$ through a common point $\pp = (u,\x)$ are called \emph{$\ell$--equivalent} at $\pp$  if  the Taylor expansions of $f_1$ and $f_2$ coincide at $\x$ up to order $\ell$.  The  class  of $\ell$--equivalent hypersurfaces to a given hypersurface $S$ at the point $\pp$ is denoted by $[S]^{\ell}_{\pp}$. The union
\begin{equation*}
J^{\ell}(n,M):=\bigcup_{\pp\in M}\{[S]^{\ell}_{\pp}\mid \text{$S$ is a hypersurface of $M$ passing through $\pp$}\}
\end{equation*}
of all these equivalence classes is the \emph{space of $\ell$--jets of hypersurfaces} $J^{\ell}(n,M)$ of $M$.
\end{definition}
Note that $J^1(n,M)=\Gr_n(TM)=\p T^*M$. From now on, when there is no risk of confusion, we let
$$
J^{\ell}:=J^{\ell}(n,M)\,.
$$
The natural projections
\begin{equation*}
\pi_{\ell,m}:J^{\ell}\stackrel{}{\longrightarrow} J^{m}\,,\quad [S]^{\ell}_{\pp}\longmapsto [S]^m_{\pp}\, ,\quad \ell>m\, ,
\end{equation*}
define a tower of bundles
$$
\dots\longrightarrow\ J^{\ell}\longrightarrow J^{\ell-1}\longrightarrow\dots\longrightarrow J^1=\mathbb{P}T^*M\longrightarrow J^0=M\,,
$$
that turn out  to be affine for $\ell\geq 2$.
For any $a^m\in J^{m}$,  the fiber of $\pi_{\ell,m}$ over $a^m$ will be indicated by the symbol
$$
J^{\ell}_{a^m}:=\pi_{\ell,m}^{-1}(a^m)\,.
$$
\subsection{Assumptions on the Lie group $G$}\label{sec:descriptions}

In what follows, unless otherwise specified, $o$ is a fixed point of  $M=G/H$ (an ``origin'') and $o^{\ell}$ a point of $J^\ell$, so that $M=G\cdot o$.  This allows us to consider the fibre $J^{\ell}_{o^{\ell-1}}$ as   a vector  space   with  the  origin  $o^{\ell}$ playing the role of zero vector.
The  group  $G$    acts naturally on each  $\ell$--jet space $J^{\ell}$:
$$
G \ni  g :   o^{\ell}= [S]^{\ell}_{o}\in J^{\ell}   \to [g(S)]^{\ell}_{g(o)}\in J^{\ell}, \,\,  o \in S\,.
$$
We observe that such an action preserves the affine structure of the fibres of $\pi_{\ell,\ell-1}$ for $\ell\geq 2$.
We denote by $H^{(\ell)}$ the stability subgroup $G_{o^{\ell}}$ in $G$ of the point $o^{\ell}$:
$$
H^{(\ell)}:=G_{o^{\ell}}\,.
$$
In this context, $G$--invariant PDEs (of order $k$) are submanifolds of $J^k$ such that $G\cdot\E=\E$.\par
In order to formulate  our  method for constructing these $G$--invariant PDEs, we have to slightly restrict the class of groups $G$ under consideration; more precisely, we are going to   assume that there exists a point $o^k\in J^k$, with $k\geq 2$, projecting to $o^{k-1}\in J^{k-1}$ such that:\par\medskip
\noindent \textbf{(A1)} the orbit
\begin{equation*}
\check{J}^{k-1} := G\cdot o^{k-1} = G/H^{(k-1)} \subset  J^{k-1}
\end{equation*}
through $o^{k-1}$ is open;\par\medskip
\noindent \textbf{(A2)} the orbit
\begin{equation}\label{eqn:Wk.orbit}
W^k:=\tau(H^{(k-1)})\cdot o^{k}\subset J^k_{o^{k-1}}
\end{equation}
through $o^k$, where
\begin{equation}\label{eqn:tau.tauk.Dmitri}
\tau :  H^{(k-1)} \to \Aff(J^k_{o^{k-1}})
\end{equation}
is the  natural affine  action  of  the    stability  subgroup     $H^{(k-1)}=G_{o^{k-1}}$   on  the  fibre  $J^k_{o^{k-1}}$,
is a vector space and   the group of translation of $W^k$, that we denote by $T_{W^k}$, is contained in $\tau(H^{(k-1)})$.
\begin{remark}
Assumption (A2) implies that  there is a point $o^{k} \in J^k$  over  the point  $o^{k-1}$   such  that   the    restriction    of   the    affine  bundle   $\pi_{k,k-1} : J^k\to J^{k-1}$ to  the orbit   $G\cdot o^{k}$   is  an  affine  subbundle  of  $\pi_{k,k-1}$ (over the base $\check{J}^{k-1}$).
\end{remark}
\subsection{A method for obtaining $G$--invariant PDEs}
Our main concern is the problem of finding   $G$--invariant PDEs   $\E \subset J^k$: if $G$ satisfies   the mild assumptions explained in Section \ref{sec:descriptions} above, 
 the affine structure of the natural bundle $\pi_{k,k-1} : J^k \to J^{k-1}$, for $k\geq 2$, will allow us to recast it   as the problem of describing  submanifolds of the fibre
$J^k_{o^{k-1}}$
that are invariant under the   affine action of  the   stability  subgroup $H^{(k-1)}$ in $G$ of $o^{k-1}$.
Moreover,  in our main  Theorem \ref{thMAIN1}, we  reduce  this last problem  to   describe $G$--invariant submanifolds  of  the    quotient  vector  space
\begin{equation}\label{eqn:V.bar}
V^k   := J^k_{o^{k-1}} / W^k\,.
\end{equation}
Since  the  action of $G$ on  $V^k$  is linear,  the above problem becomes a  standard problem  in the  theory of invariants  of  a linear Lie  group, which  is  much  simpler  than  the initial one---that is, the problem of  describing  the invariants  for  a non--linear  action   of    the Lie  group $G$  on  the manifold  $J^k$. As an application of the main Theorem \ref{thMAIN1}, we  solve  this  problem  in   the  case  when  the homogeneous  manifold $M = G/H$  defines  either the Euclidean or the conformal  geometry (in  the  sense of F. Klein).
We stress that the approach we propose does not rely on machine--aided computations and at the same time sheds light on some geometric properties of the $G$--invariant PDEs.

\par

\subsection{Structure of the paper}\label{sec.structure}

In Section \ref{SecMainRes} we   recall  some basic definitions concerning the geometry of the spaces $J^{\ell}=J^{\ell}(n,M)$ of $\ell$--order jets  of hypersurfaces of an $(n+1)$--dimensional manifold $M$, as well as  of their subbundles, that   are systems of PDEs in one unknown variable.

In Section \ref{secConstructingGinvPDEs}, under the assumptions (A1) and (A2) of Section \ref{sec:descriptions}, we prove  the main Theorem \ref{thMAIN1},
which reduces the  construction   of  $G$--invariant PDEs $\mathcal{E} \subset   J^k$ to   the description  of  hypersurfaces  of $V^{k}$ that are  invariant  with  respect  to  the linear action  of  the stability  group $H^{(k-1)}$.


In Section \ref{secCasEucl} we carefully examine  the case when $M =  \EE^{n+1}$ is the   $(n+1)$--dimensional Euclidean space,  considered  as    the  homogeneous  space   $\EE^{n+1} = \SE(n+1)/\SO(n+1)$ of  the  group  of    orientation--preserving  motions.
It would be sensible to stress that the main purpose of discussing here the Euclidean case is that of  testing  the  results of Section \ref{secConstructingGinvPDEs} on a particularly simple and well--known ground.
%
%

In Section \ref{secCasConf} we pass to the case  when $M$ is the conformal sphere $\SSS^{n+1}$, that is a homogeneous space of the special orthogonal  group $\SO(1,n+2)$, called also the \emph{M\"obius group}:  we obtain the conformally invariant PDEs in terms of invariants of a certain space of traceless quadratic forms. For instance, in Section \ref{sec:conf.local}, in the case of two independent variables, we see that the unique $\SO(1,4)$--invariant PDE is expressed in terms of the Fubini's conformally invariant first fundamental form.\par

\par

We would like to underline that the invariant PDEs we found are expressed as the zero set of a function of the invariants (of a certain order) of the considered group. This, of course, does not guarantee that the so--obtained PDE is a scalar one, as the zero set of a real--valued function is not always a codimension--one submanifold. In fact, this happens for the conformal group, as described in .

A key clarification is in order. The main output of the applicative  Sections  \ref{secCasEucl}  and \ref{secCasConf} are invariant polynomials: their zero sets will then provide us with the $G$--invariant PDEs, understood as hypersurfaces, that were predicted by the main theoretical result, Theorem \ref{thMAIN1}, but only in the case when the aforementioned zero set is a submanifold of codimension one (see on this concern the   example treated in Section \ref{sec:conf.local}). In this perspective, in  Sections  \ref{secCasEucl}  and \ref{secCasConf} we produce more invariant objects than those given by Theorem \ref{thMAIN1}, and indeed the corresponding Theorem \ref{th.euc.main} and Theorem \ref{th.conf.main} state that among the zero sets there are the PDEs anticipated by Theorem \ref{thMAIN1}---they do not claim that these zero sets account for all such PDEs. Such a discrepancy disappears in the complex case, and this is the main reason why in the already cited work \cite{MR3904635}  the authors worked with complex Lie groups form the outset.

\section{The  affine structure of the bundles of jet spaces}\label{SecMainRes}

\subsection{Jets of hypersurfaces of $M$}\label{subJetsSubs}

The space $J^{\ell}$ has a natural structure of smooth manifold: one way to see this is to extend the   local coordinate system \eqref{eqn:coord.u.x} on $M$
to  a  coordinate system
\begin{equation}\label{eqn:jet.coordinates}
(u,\x,\ldots,u_i,\ldots, u_{ij}, \ldots, u_{i_1 \cdots i_l},\ldots) = (u,x^1,\dots,x^n,\ldots,u_i,\ldots, u_{ij}, \ldots, u_{i_1 \cdots i_l},\ldots)
\end{equation}
on $J^{\ell}$, where each coordinate function\footnote{The $u_{i_1 \cdots i_k}$'s are symmetric in the lower indices.} $u_{i_1 \cdots i_k}$, with $k\leq \ell$,  is unambiguously defined by the rule
\begin{equation}\label{eqDefHigCoordJets}
u_{i_1 \cdots i_k}\left([S_f]^{\ell}_{\pp}\right)= \partial^k_{i_1 \cdots i_k}f(\x)\, ,\quad \pp=(u,\x)\,,\quad k\leq \ell\,.
\end{equation}
In formula \eqref{eqDefHigCoordJets} above the symbol $\partial_i$ denotes the partial derivative $\partial_{x^i}$, for  $i= 1,  \ldots, n$; we recall that the hypersurface $S=S_f$ is the graph of the function $u=f(\x)$ and, as such, it is   admissible for the chart $(u,\x)$.\par
%
The $\ell$--lift of $S$ is defined by
\begin{equation}\label{eq:jet.ext}
S^{(\ell)}:=\{[S]^{\ell}_{\pp}\,\,|\,\,\pp\in S\}\,.
\end{equation}
It is an $n$--dimensional submanifold of $J^{\ell}$. If $S=S_f$ is the graph of  $u=f(\x)$, then $S_f^{(\ell)}$ can be naturally parametrized as follows:\footnote{We stress once again that a switch has occurred between the first and the second entry, with respect to a more standard literature.}
\begin{equation}\label{eqn:lift.local}
\left(u=f(\x),\x,\dots u_i=\frac{\partial f}{\partial x^i}(\x) ,\dots  u_{ij}=\frac{\partial^2 f}{\partial x^i\partial x^j}(\x) , \dots \right)\,.
\end{equation}
\begin{remark}
In the case $M$ is a fibre bundle $\pi:M\to B$ with $n$--dimensional fibres, one can define the space of $\ell$--jets $J^l\pi$ of $\pi$ as the space of $\ell$--jets of the graphs of local sections of $\pi$. The space $J^{\ell}\pi$ is an open dense subset of $J^{\ell}=J^{\ell}(n,M)$. In the case that $\pi:\R\times N\to N$ is the trivial bundle, $\dim N=n$, then $J^{\ell}\pi$ coincides with the space of $\ell$--jets of functions on $N$.
\end{remark}

\subsection{The tautological bundle and the higher order contact distribution on $J^{\ell}$}
In this section,   to not overload the notation, we denote a point $[S]^{\ell}_{\pp}\in J^{\ell}$ by $a^{\ell}$.
The next  lemma is well known.
\begin{lemma}\label{lemma.Dmitri}
Any point $a^{\ell}=[S]^{\ell}_{\pp}\in J^{\ell}$ canonically defines the $n$--dimensional subspace
\begin{equation}\label{eq:Dmitri.1}
T_{a^{\ell-1}}S^{(\ell-1)}\subset T_{a^{\ell-1}}J^{\ell-1}\,,\quad a^{\ell-1}=\pi_{\ell,\ell-1}(a^{\ell})\,.
\end{equation}
\end{lemma}
\begin{definition}\label{defTautBundle}
The \emph{tautological} rank--$n$ vector bundle $\TT^{\ell} \subset \pi_{\ell,\ell-1}^* (TJ^{\ell-1})$   is the bundle over $J^{\ell}$ whose fiber over the point $a^{\ell}$ is given by  \eqref{eq:Dmitri.1}, i.e.,
\begin{equation*}
\TT^{\ell}=\left\{ (a^{\ell},v)\in J^{\ell}\times TJ^{\ell-1}\,\,|\,\,v\in  T_{a^{\ell-1}}S^{(\ell-1)}\right\} \,.
\end{equation*}
\end{definition}
The (truncated) total derivatives
\begin{equation}\label{eqn:total.derivatives}
D_i^{(\ell)}:=\partial_{x^i}+\sum_{k=1}^{\ell}\sum_{j_1\leq\dots\leq j_{k-1}}u_{j_1\dots j_{k-1}\,i}\,\partial_{u_{j_1\dots j_{k-1}}}\,,\quad i=1\dots n\, ,
\end{equation}
constitute a local basis of the bundle $\TT^{\ell}$.\par

%
By considering the preimage of the tautological bundle on $J^{\ell}$ via the differential $d\pi_{\ell,\ell-1}$ of the canonical projection $\pi_{\ell,\ell-1}$, we get a distribution on $J^{\ell}$,   denoted by $\CC^{\ell}$:
\begin{equation*}
\CC^{\ell}:=(d\pi_{\ell,\ell-1})^{-1}\TT^{\ell}\, .
\end{equation*}
\begin{definition}
 $\CC^{\ell}$ is called the \emph{$\ell\Th$ order contact structure} or \emph{Cartan distribution} (on $J^{\ell}$).
\end{definition}
Above formula, applied to a particular point $a^{\ell}$ of $J^{\ell}$ that projects on  $a^{\ell-1}\in J^{\ell-1}$, reads
\begin{equation}\label{eqDefCartanDistr}
\mathcal{C}^{\ell}_{a^{\ell}}=(d\pi_{\ell,\ell-1})^{-1}(T_{a^{\ell-1}}S^{(\ell-1)})=\TT^{\ell}_{a^{\ell}} \oplus T^v_{a^{\ell}}J^{\ell}\,,
\end{equation}
where
$
T^vJ^{\ell}:=\ker (d\pi_{\ell,\ell-1})
$
is the \emph{vertical subbundle} of $TJ^{\ell}$. \par
Distribution $\mathcal{C}^{\ell}$ can be considered as a ``higher order contact structure'' \cite{KrasilshchikLychaginVinogradov:GJSNPDEq,MR2352610}  since, for $\ell=1$,   if $(u,x^i,u_i)$ is a chart on $J^1$, then
\begin{equation}\label{eq:contact.form.conormal}
\mathcal{C}:=\mathcal{C}^1=\ker(\theta)\,,\quad\text{where}\quad\theta=du-u_idx^i
\end{equation}
is a contact form. For $\ell>1$,   the    planes of the $\ell\Th$  order contact structure  $\mathcal{C}^{\ell}$ are the kernels of the following system of $1$--forms (Pfaff system):
$$
\theta = du- u_i dx^i\,,\,\,  \theta_{i_1\dots i_k} = du_{i_1\dots i_k} -  u_{i_1\dots i_k h} dx^h,  \,\,\,\, k<\ell,\,\,  i =1, \ldots, n\, .
$$

\subsection{The affine structure of the  bundles   $J^{\ell}\to J^{\ell-1}$ for $\ell\geq 2$}\label{secSezioneDedicataADmitri}

In this section we describe the affine structure of the bundles $\pi_{\ell,\ell-1}:J^{\ell}\to J^{\ell-1}$, $\ell\geq 2$.
In  order to state Proposition \ref{prop:affine.bundles} below, we need to introduce yet another bundle over $J^1$, tightly  related with the tautological bundle defined before: according to   Definition \ref{defTautBundle}, the tautological bundle $\TT:=\TT^1$ is the vector bundle over $J^1$ defined by
\begin{equation*}
\TT_{[S]_{\pp}^1}:=\TT^1_{[S]_{\pp}^1}=T_{\pp}S\,  .
\end{equation*}
\begin{definition}
The \emph{normal bundle} $\NN$ is the following line bundle over   $J^1$:
\begin{equation*}
\NN_{[S]_{\pp}^1}:=N_{\pp}S:= T_{\pp}M\big/T_{\pp}S\,.
\end{equation*}
\end{definition}
%
%
%
To simplify notations, we denote by $\partial_u$ the equivalence class $\partial_u\,\mathrm{mod}\,\TT$.
%

Next Lemma is well known (see for instance \cite{KrasilshchikLychaginVinogradov:GJSNPDEq,MR989588}): it describes the vertical subbundle $T^vJ^{\ell}$ of $J^{\ell}$ in terms of the bundles $\TT$ and $\NN$.
\begin{lemma}\label{lemma.ker.vector}
We have that
\begin{equation*}
T^vJ^{\ell} \simeq \pi_{\ell,1}^* (S^{\ell}\TT^* \otimes \NN)\, .
\end{equation*}
In local coordinates \eqref{eqn:jet.coordinates}, the above isomorphism  gives (up to a constant) the bijection
\begin{equation}\label{eqdxdxdxdu}
\frac{\partial}{\partial u_{i_1\cdots i_{\ell}}}\longleftrightarrow dx^{i_1}\odot\cdots\odot dx^{i_{\ell}}\otimes \partial_u\, ,
\end{equation}
where $\odot$ is the symmetric product.
\end{lemma}
The next proposition  is also well known (\cite{KrasilshchikLychaginVinogradov:GJSNPDEq,MR989588}) and it is crucial for our purposes.
\begin{proposition}\label{prop:affine.bundles}
For $\ell\geq 2$ the bundles $J^{\ell}\to J^{\ell-1}$ are affine bundles modeled by the vector bundles $\pi_{\ell-1,1}^* (S^{\ell}\TT^* \otimes \NN)$. In particular, once  a  chart $(u,\x)$ has been fixed, the choice of a point $[S]^{\ell}_{\pp}$ (the origin) defines the identification  $J^{\ell}_{[S]^{\ell-1}_{\pp}}$ with $S^{\ell}T^*_{\pp} S$.
\end{proposition}
Below we give a sketch of the proof based on the action of a generic element $v$ of $\big(\pi_{\ell-1,1}^*(S^{\ell}\TT^* \otimes \NN)\big)_{[S]^{\ell-1}_{\pp}}$ on the fibre $J^{\ell}_{[S]^{\ell-1}_{\pp}}$ (for more details see for instance \cite{MR989588}).
Without loss of generality, we can choose $\pp=o=(0,\boldsymbol{0})$ and a chart $(u,\x)$ admissible by $S$, so that $S=\{u=f(\x)\}=S_f$.
The aforementioned element can be written as
\begin{equation*}
v=v_{i_1\cdots i_{\ell}}\,dx^{i_1}\odot\cdots\odot dx^{i_{\ell}}\otimes \partial_u \in S^{\ell}(T_{\pp}^*S)\otimes\NN_{[S]^1_{\pp}}
\end{equation*}
(see Lemma \ref{lemma.ker.vector}); its   action
on $[S]^{\ell}_{\pp}=[S_f]^{\ell}_{\pp}$ is given by  $v:[S_f]^{\ell}_{\pp} \to [S_g]^{\ell}_{\pp}$ where
\begin{equation*}
g(\x)=f(\x)+\frac{1}{\ell !}v_{i_1\cdots i_{\ell}}x^{i_1}\cdot\cdots\cdot x^{i_{\ell}}\, ,
\end{equation*}
i.e., the $\ell$--order derivatives $f_{i_1\cdots i_{\ell}}$ are sent to $g_{i_1\cdots i_{\ell}}=f_{i_1\cdots i_{\ell}}+v_{i_1\cdots i_{\ell}}$.

\begin{warning}
 From now on, the symmetric product $dx^i\odot dx^j$ (resp. $\partial_{x^i}\odot\partial_{x^j}$) will be denoted simply by $dx^idx^j$ (resp. $\partial_{x^i}\partial_{x^j}$).
\end{warning}

\section{A  general construction  of  $G$--invariant  PDEs     on  a  homogeneous  manifold  $M = G/H$} \label{secConstructingGinvPDEs}

As above, let $M = G/H  =  G\cdot o$, $o\in M$, be  an  $(n+1)$--dimensional  homogeneous manifold. Recall (see Section \ref{sec:descriptions}) that $G$ acts on each  jet space $J^\ell=J^{\ell}(n,M)$. We recall the following definitions.

\begin{definition}
A  \emph{system}  of  $m$ PDEs of order $k$ is  an $m$--codimensional submanifold   $\E  \subset  J^k$. A  \emph{solution}  of    the  system   $\E$  is  a  hypersurface   $S \subset M$    such  that $[S]^k_{\pp} \in \E$  for  all $\pp \in  S$. The system $\E$ is called   \emph{$G$--invariant} if  $G \cdot\E = \E$.
 \end{definition}
The  aim of  this  section  is  to  reduce  the  problem of  describing    $G$--invariant scalar PDEs  (i.e., when $m=1$) of order $k$,  that  is, $G$--invariant  hypersurfaces  in  $J^k$, to the problem of  describing     hypersurfaces   in a certain  vector  space,  invariant  under   the linear  action   of   the   stability  subgroup $H^{(k-1)}$ of  the point  $o^{k-1}\in  J^{k-1}$.\par
%
Below we give some definitions together with some preliminary lemmas that are important to state the main Theorem \ref{thMAIN1}.
%
%
%
\begin{definition}\label{def.Dmitri}
A  homogeneous  manifold   $M = G/H$  is called \emph{$k$--admissible}  for  $k \geq 2$ if assumptions (A1) and (A2) of Section \ref{sec:descriptions} are satisfied.
\end{definition}
\begin{definition}\label{defFidHyp}
A  hypersurface $S \subset  M$  through   the point  $o$ that is homogeneous  with respect to a  subgroup  of $G$ for which $o^{k-1}=[S]^{k-1}_o$ and $o^k=[S]^k_o$ satisfy (A1) and (A2) of Section \ref{sec:descriptions} is  called  a \emph{fiducial hypersurface}.
\end{definition}
\begin{remark}
It is worth stressing that the main theoretical result, Theorem \ref{thMAIN1} below,  does not require  the existence of a fiducial hypersurface, whereas   its applications to the particular cases discussed later in Sections \ref{secCasEucl} and \ref{secCasConf}, become geometrically more trasparent thanks to an obvious choice of a fiducial hypersurface; in particular, the latter allows constructing  a preferred coordinate system in which the so--obtained invariant PDEs look particularly simple. The authors did not deepen the problem of existence of a fiducial hypersurface for \emph{all} $k$--admissible homogeneous manifolds $M=G/H$.
\end{remark}
Below we state an elementary lemma concerning affine subgroups that are semidirect product of their  linear (canonically associated) group and a translational one.\par
Let $V$ be a vector space, treated as an affine space  with origin $o$. Let $H$ be a subgroup of $\Aff(V)=V\rtimes\GL(V)$. Assume that $W:=H\cdot o$ is a vector subspace of $V$ and that the corresponding group of translations $T_W$ is a subgroup of $H$. Then
\begin{equation}\label{eqn:H.T.L}
H=T_W\rtimes L_H\, ,
\end{equation}
where $L_H$ is the linear subgroup of the stabilizer of the origin $o$. Since $T_W$ is a normal subgroup, we have that $L_H\cdot W=W$. Denote by $U$ a subspace complementary to $W$, that is
\begin{equation}\label{eqn:VWU}
V=W\oplus U\,,
\end{equation}
so that the natural projection $p:V\to V/W$ defines an identification  $p|_U:U\to V/W$. In view of \eqref{eqn:H.T.L},  an element $h\in H$ can be uniquely presented as
\begin{equation}\label{eqn:h.T.wh}
h=T_{w(h)}\cdot L_h\in H\, ,
\end{equation}
  so that its induced action
 on $V/W$ corresponds to the linear action given by  $h:u\to \overline{L}_h\cdot u$, where $L_h$, in terms of decomposition \eqref{eqn:VWU}, is
\begin{equation}\label{eqn:Lh}
L_h=
\left(
\begin{array}{cc}
* & *
\\
0 & \overline{L}_h
\end{array}
\right)\,.
\end{equation}

\begin{lemma}\label{lemma.aff.banale}
Let $H$ as in \eqref{eqn:H.T.L} and $V$ as in \eqref{eqn:VWU}. Then there exists a $1$--$1$ correspondence between  $\overline{L}_H$--invariant hypersurfaces $\overline{\Sigma}\subset U=V/W$ and (cylindrical) $H$--invariant hypersurfaces $\Sigma=W+\overline{\Sigma}$ in $V$.
\end{lemma}
\begin{proof}
Let $\overline{\Sigma}\subset U$ be a $\overline{L}_H$--invariant hypersurface. Then $\Sigma=W+\overline{\Sigma}$ is a  $H$--invariant hypersurface of $V$ since, for each $w+u\in \Sigma=W+\overline{\Sigma}$, in view of \eqref{eqn:h.T.wh} and \eqref{eqn:Lh},
$$
h(w+u)=L_h(T_{w(h)}(w+u))=L_h(w+w(h)+u)=L_h(w+w(h))+\overline{L}_h(u)\in W+\overline{\Sigma}\, .
$$
Conversely, if $\Sigma\subset V=W\oplus U$ is an $H$--invariant hypersurface, then  $T_W\cdot\Sigma=W+\Sigma\subset\Sigma$, i.e., $\Sigma$ is a cylindrical hypersurface and the quotient $\overline{\Sigma}=\Sigma\cap U$ is an $\overline{L}_H$--invariant hypersurface in $U=V/W$.
\end{proof}
Recall now that, if $S$ is a fiducial hypersurface in the sense of Definition \ref{defFidHyp}, then
$$
o^{\ell}:= [S]^{\ell}_o \in J^{\ell}
$$
plays the role of the origin in $J^{\ell}$. Furthermore, we have the following identification (see Proposition \ref{prop:affine.bundles}):
$$
J^{\ell}_{o^{\ell-1}}= S^{\ell}(T_{o}^*S)\otimes N_oS\,.
$$
If we represent  the  fiducial hypersurface $S$ as a graph $S_f$, then we can write
 (see again Proposition \ref{prop:affine.bundles}):
\begin{equation}\label{eqn:identification}
J^{\ell}_{o^{\ell-1}}= S^{\ell}(T_{o}^*S_f)\,.
\end{equation}
From now on,  we will  use  this identification.
%

%

For  any  $h \in  H^{(\ell-1)}$ we   may   decompose the  affine   transformation   $\tau(h)\in \Aff(J^{\ell}_{o^{\ell-1}})$  (see \eqref{eqn:tau.tauk.Dmitri}) into the product of its  linear  part  $A_h \in  \GL(S^{\ell}(T_{o}^*S_f))$, which is the stabilizer of the point $o^{\ell}$, and  the   translation  $T_h$ along the vector $\tau(h)(o^{\ell})$, i.e.,  $$\tau(h)= T_h \cdot A_h\, .$$
%
Now we assume that the homogeneous manifold $M=G/H$ is $k$--admissible.
Thus, condition (A2) shows that
$$
W^k = \tau(H^{(k-1)}) \cdot o^k  = \{  T_h \cdot o^k, \,  h \in  H^{(k-1)} \}
$$
is  a vector  subspace of  $J^k_{o^{k-1}}$ and, furthermore, that any element $\tau(h)\in\tau(H^{(k-1)})$ can be decomposed as follows (see \eqref{eqn:h.T.wh}):
$$
\tau(h)=T_{w(h)}\cdot L_h\,,\quad T_{w(h)}\,,\,\,L_h\in\tau(H^{(k-1)})\, .
$$
Hence,
\begin{equation}\label{eqn:decomp.tau.H}
\tau(H^{(k-1)})=T_{W^k}\rtimes L_{H^{(k-1)}}\, ,
\end{equation}
where $L_{H^{(k-1)}}$ is the stabilizer of $o^k$.
Applying Lemma \ref{lemma.aff.banale} to the affine subgroup $\tau(H^{(k-1)})\subset \Aff(J^k_{o^{k-1}})$ we get the following corollary.
\begin{corollary}\label{cor.1.1.corresp}
Let $M=G/H$ be a $k$--admissible homogeneous manifold.
Then there exists a $1$--$1$ correspondence between $L_{H^{(k-1)}}$--invariant hypersurfaces  $\overline{\Sigma}\subset J^k_{o^{k-1}}/W^k$ and (cylindrical) $\tau(H^{(k-1)})$--invariant hypersurfaces $\Sigma=p^{-1}(\overline{\Sigma})\subset J^k_{o^{k-1}}$,  where
\begin{equation}\label{eqn:p.proj}
p: J^k_{o^{k-1}}\to J^k_{o^{k-1}}/W^k
\end{equation}
is the natural projection.
\end{corollary}



\begin{lemma}\label{lemLemmaLemmino}
Let  $\pi : P \longrightarrow B$  be  a  bundle. Assume that a Lie group $G$ of automorphisms  of  $\pi$,  such  that $B =  G/H$,  acts  transitively on  $B$,  where $H$ is  the  stabilizer of a point  $o \in B$.  Then:
\begin{itemize}
\item [i)] any $H$--invariant function $F$ on $P_o:=\pi^{-1}(o)$ extends to a $G$--invariant  function  $\widehat{F}$ on  $P$ (where  $\widehat{F}(gy) =F(y)$ for  $y \in  P_o$   and     $g \in  G$), and this is a $1$--$1$ correspondence;
\item[ii)] any $H$--invariant   hypersurface   $\Sigma$  of  the  fiber  $P_o$ extends to  a $G$--invariant   hypersurface   $\mathcal{E}_{\Sigma} := G \cdot\Sigma$  of  $P$, and this is a $1$--$1$ correspondence.
\end{itemize}
\end{lemma}

\begin{proof}
The  stabilizer $H$  acts on $P_o$  and  we may identify   $\pi$  with  the  homogeneous  bundle  $\pi:  G \times_H P_o \to  B = G/H$  associated  with  the principal bundle  $G \to G/H$  and  the  action  of $H$ on $P_o$. Recall  that   $ G \times_H P_o$  is the orbit  space   of  the manifold  $G\times P_o$  with  respect  to the  action of $H$,  given  by
$$
H \ni h : (g,y) \mapsto  (gh^{-1}, hy)\, .
$$
i) The  restriction   to  $P_o$  of  a $G$--invariant  function $F$  on  $G \times_H P_o $ is identified  with  a left--invariant  function on $G\times P_o$ ,  which is   also $H$--invariant,  that  is a  function $  F(g,y)$  such  that
$F(g' g,y) = F(gh^{-1}, hy) = F(g,y)$ for  all $g,g' \in G$, $h \in H$, $y \in P_o$. Such a function  is   identified  with an  $H$--invariant  function  on  $P_o$.\\
ii)  An  $H$--invariant hypersurface $\Sigma \subset P_o$ defines  a  $(G \times H)$--invariant  hypersurface
$G \times \Sigma \subset  G\times P_o $.  It  projects  onto the $G$--invariant  hypersurface  $\E_{\Sigma}=G\cdot \Sigma$  in $P = G \times_H P_o$.
\end{proof}
%
%
Corollary \ref{cor.1.1.corresp}, together with Lemma \ref{lemLemmaLemmino}, applied to the bundle $\pi_{k,k-1}$, implies  the  following   theorem,  which is  the main  result of  this  section.
\begin{theorem}\label{thMAIN1}
Let $M=G/H$ be a $k$--admissible homogeneous manifold (see Definition \ref{def.Dmitri}). Then there is  a natural  $1$--$1$   correspondence  between   $L_{H^{(k-1)}}$-- invariant hypersurfaces $\overline{\Sigma}$  (see also \eqref{eqn:decomp.tau.H}) of $J^k_{o^{k-1}}/W^k$
and  $G$--invariant  hypersurfaces  $\E_{\overline \Sigma}:= \E_{p^{-1}(\overline{\Sigma})}=   G \cdot p^{-1}(\overline{\Sigma})$  of  $J^k=J^k(n,M)$, where $p$ is the natural projection \eqref{eqn:p.proj}.
\end{theorem}
%
In view of the discussions we did so far, taking into account Theorem \ref{thMAIN1}, we  get  the  following   strategy  for    constructing     $G$--invariant  PDEs   imposed on  the hypersurfaces  of   a  $k$--admissible homogeneous manifold $M = G/H$:
%
%

\begin{enumerate}
\item    calculate   the orbit   $W^k = \tau(H^{(k-1)})\cdot o^{k} $   and decompose $\tau(H^{(k-1)})$ accordingly  to \eqref{eqn:decomp.tau.H};
\item   describe $L_{H^{(k-1)}}$--invariant  hypersurfaces $\overline{\Sigma} \subset {V}^k=J^{k}_{o^{k-1}}/W^k$;
\item   write  down the  $G$--invariant   equations   $\E_{\overline \Sigma}  =  G \cdot p^{-1}(\overline{\Sigma})$ in  coordinates \eqref{eqn:jet.coordinates}.
\end{enumerate}

In  the   next  sections  we implement  this strategy  for the  Euclidean and the conformal  space.

\begin{remark}
 It may be worth noticing that, in order  to have the above strategy work, only the assumptions (A1) and (A2) on the origins $o^k$ and $o^{k-1}$ are strictly necessary. The fiducial hypersurface (see Definition \ref{defFidHyp}) is just a way to introduce the aforementioned origins in a  more tangible geometric way: indeed, they reflect the presence of a ``reference object'' to which we compare the jet of a generic hypersurface. The hypothesis of existence of such a fiducial hypersurface, which we even require to  be homogeneous with respect to a subgroup of $G$, is by no means restrictive: we will see below that in both the Euclidean and the conformal case, such a hypersurface clearly exists.
\end{remark}


\section{Invariant PDEs   for  hypersurfaces of $\EE^{n+1}$}\label{secCasEucl}

In this section  we  assume  that  $M= \EE^{n+1} =G/H  = \SE(n+1)/\SO(n+1) $ is    the Euclidean    $(n+1)$--dimensional space,  considered  as      homogeneous    space of  the  group
$$G=\SE(n+1)=\R^{n+1}\rtimes \SO(n+1)$$
of orientation--preserving motions.
We  will find $\SE(n+1)$--invariant PDEs following  the approach described in  the  previous  section.
\subsection{Geometry  of  the   jet  spaces  $J^{\ell}(n , \EE^{n+1}) $
  for  $\ell =1,2$ in terms of $G=\SE(n+1)$}
In this section we    give  a   description   of  the   jet  spaces $J^{\ell}=J^{\ell}(n , \EE^{n+1}) $ in terms of the Lie group action of $\SE(n+1)$ and  prove  that   the   Euclidean homogeneous  space  $\EE^{n+1}  = G/H$  is a $2$--admissible manifold.

Denote  by $(u,\x)=(u,x^1,\ldots, x^n)$  the  Euclidean  coordinates  associated  to  an orthonormal  frame\linebreak  
 $\{e_0,e_1,\ldots, e_n \}$   at  a point $o  \in \EE^{n+1}$   and identify   $\EE^{n+1}$  with  the  arithmetic  vector  space
 $\R^{n+1}$  of  coordinates.  In particular, $o=(0,0,\dots,0)\in\R^{n+1}$.   Using  the standard euclidean metric,   we identify   covectors  with  vectors.
We fix  the  hyperplane
$$
S_0 = \langle e_1, \ldots, e_n\rangle
$$
through  the origin $o$.    In   the  coordinates $(u,\x)$, $S_0$ is defined by  the   function    $u = f(\x) = 0$ and  all  the coordinates $u_{i_1\cdots i_{\ell}}$ of its lift  $S_0^{(l)}$ (see \eqref{eq:jet.ext}--\eqref{eqn:lift.local}) are identically  zero. Below we will show that $S_0$ is a fiducial hypersurface according to Definition \ref{defFidHyp}: to this end   we  denote  by $o^k=[S_0]^k_o$  the    distinguished point of $J^k$ over  the origin  $o \in  \EE^{n+1}$.   In  particular  $o^1 \in J^1$,   considered  as    a  tangent  vector  space, will be  denoted   by   \begin{equation}\label{eqn:V.Dmitri}
V := T_oS_0 = \langle e_1, \dots, e_n\rangle\,.
\end{equation}

\begin{proposition}\label{PropFidHypCasoEuclid}
The   homogeneous  space    $\EE^{n+1} = G/H$ is  $2$--admissible. More precisely,  $J^1 =  G \cdot o^1 =  G/H^{(1)}$,  where $H^{(1)} = \OOO(n)$ is  the  subgroup of  $H =  \SO(n+1)$   which preserves  the  vector  $e_0$  up  to  the  sign.   The   hyperplane  $S_0$  is  a  fiducial  hypersurface.
\end{proposition}
\begin{proof}
The  stability   subgroup  of  the origin $o \in  \EE^{n+1}$  is  $\SO(n+1)$.   It  acts  transitively on  the Grassmannian  $\Gr_{n}(T_o\EE^{n+1} ) \simeq  \p T^*_o \EE^{n+1}$,  which is  the  fiber  of  the  bundle
$J^1=\p T^*\EE^{n+1} \to J^0 =  \EE^{n+1}$  over  the point  $o$.  The  stability  subgroup  of  the point $o^1= [S_0]^1_o$
 is  the   subgroup   $H^{(1)} = \OOO(n)$ of  $H = \SO(n+1)$: it preserves  $e_0\in V^\perp$ up  to  the   sign.
Hence   $J^1   =  G/ H^{(1)}$ and the condition (A1) of Section \ref{sec:descriptions} is satisfied.\\
 Let   $S_f = \{u = f(\x) \}$  be  a hypersurface  through   the origin  $o$  with unit normal vector
 $e_0$  at $o$, so that $[S_f]^1_o=V$.  Then  the  second jet $[S_f]^2_0$    has  coordinates  $\x=0$, $u=0$, $u_i =\tfrac{\partial f}{\partial x^i}(0)= 0$, $u_{ij} = \tfrac{\partial^2 f }{\partial x^i\partial x^j}(0)$   (see also the notation \eqref{eqDefHigCoordJets}) and it  is identified  with  the  second  fundamental  form  $\beta =  u_{ij}dx^idx^j\in S^2V^*$   of  the  hypersurface $S_f$  at  $o$.
Recalling that the fiber  $J^2_{o^1}$  with  the origin  $o^2$ is identified  with the space $S^2V^*$ (see \eqref{eqn:identification}), the natural action  $\tau$ of $\OOO(n)$ on $J^2_{o^1}\simeq S^2V^*$ is:
$$
\tau(B)  :  \beta\in S^2V^* \mapsto  B^t \beta B \in S^2V^* \,, \quad B\in  H^{(1)} = \OOO(n)\,.
$$
In particular, this  shows   that   $G$  has  no open orbits  in $J^2$. Furthermore, since $\tau(H^{(1)})$ is a linear group, its translational part $W^2$ is trivial and the condition (A2) of Section \ref{sec:descriptions} is satisfied, so that $\EE^{n+1}$  is    a  $2$--admissible   homogeneous  space and $S_0$ a  fiducial hypersurface.
\end{proof}
Note that, in this case,
\begin{equation}\label{eqn:euc.V2}
V^2=J^2_{o^1}\big/ W^2=J^2_{o^1}=S^2V^*\,.
\end{equation}

\begin{corollary}\label{cor.Dmitri.1}
The   stability   subgroup   $H^{(2)}$   of  the  point $o^2 =  [S_f]^2_o$  of    a  hypersurface  $S_f$ through  the point  $o$  with $[S_f]^1_o  = o^1 $  is  the  subgroup  of  $H^{(1)} = \OOO(n)$ that preserves  the  second  fundamental  form $\beta$ of $S_f$,  that is
$$
H^{(2)} =   \{ B\in \OOO(n) \,\,|\,\,\tau(B)(\beta) = \beta\}\,.
$$
\end{corollary}

\subsection{Construction of $\SE(n+1)$--invariant PDEs}\label{secEuclDmitri}
In the considered case, the construction of  $\SE(n+1)$--invariant PDEs, in view of Theorem \ref{thMAIN1}, reduces to the description of $\tau(\OOO(n))\simeq\OOO(n)$--invariant hypersurfaces in $J^2_{o^1}=S^2V^*\simeq S^2\R^n$ (see \eqref{eqn:V.Dmitri} and the identification \eqref{eqn:identification}).

Denote by $k_1, \dots , k_n$ the eigenvalues of  the  shape operator   $A=g^{-1} \circ \beta$ (principal  curvatures), where $g$ is the restriction to \eqref{eqn:V.Dmitri} of the euclidean metric of $\EE^{n+1}$. Any $\OOO(n)$--invariant polynomial on $S^2V^*$ is a polynomial $F(\sigma_1, \dots , \sigma_n)$
 where  $\sigma_i,\,  i =1, \dots, n$, are the elementary  symmetric functions of  the   principal   curvatures $k_1, \dots, k_n$
or, equivalently, of the symmetric functions
$$
\tau_m=\tr{(A^m)}\, .
$$

An invariant polynomial $F=F(\tau_1,\dots,\tau_n)$ defines the $\OOO(n)$--invariant algebraic hypersurface
\begin{equation}\label{eqn:euc.F.}
\Sigma=\overline{\Sigma}=\{F=0\}\subset S^2V^*
\end{equation}
(see also \eqref{eqn:euc.V2}). The associated hypersurface $\E=\E_{\Sigma}=\SE(n+1)\cdot {\Sigma}\subset J^2$ is an $\SE(n+1)$--invariant hypersurface, that defines a second order PDEs polynomial in the second derivatives.
Solutions to PDE $\E_{\Sigma}$ are hypersurface $S\subset M$ whose shape operator $A_{\pp}$ satisfies, $\forall\,\pp\in S$, the equation
\begin{equation}\label{eqn:euc.final.form.F}
F({\tau_1}_{\pp},\dots,{\tau_n}_{\pp})=F\left(\tr(A_{\pp}),\tr(A_{\pp}^2),\dots,\tr(A_{\pp}^n)\right)=0\,.
\end{equation}

\begin{remark}
More generally, any function $F(\tau_1,\dots, \tau_n)$ defines the $\SE(n+1)$--invariant PDE $F=0$.
\end{remark}

Finally, Theorem \ref{thMAIN1} implies the following one.
\begin{theorem}\label{th.euc.main}
Any second order $\SE(n+1)$--invariant PDE for  hypersurfaces  $S_f=\{u=f(\x)\}$ in  $\EE^{n+1}$ which is a polynomial in the second order derivatives of $f$
is of the form \eqref{eqn:euc.final.form.F}, where $F$ is a polynomial of $n$ variables.
\end{theorem}

\subsubsection{Description  of  $\SE(n+1)$--invariant  PDEs in local  coordinates}\label{sec.euc.local}
Below we  write  down, in coordinates \eqref{eqn:jet.coordinates},    $\SE(n+1)$--invariant PDEs $\E  \subset  J^2=J^2(n, \EE^{n+1}) $ for  a hypersurface  $S_f = \{  u = f(\x)\}$. The first fundamental form $g$ of $S_f$ is given by
$$
g=g_{ij}dx^idx^j=\big(\delta_{ij} + u_iu_j\big)dx^idx^j\, ,
$$
so that
$$
g^{-1}=g^{ij}\partial_{x^i}\partial_{x^j}=\left(\frac{\det(g)\delta_{ij}-u_iu_j}{\det(g)}\right)\partial_{x^i}\partial_{x^j}
\,,\quad \det(g)=1+\sum_{h=1}^n u_i^2\,.
$$
%
The  second  fundamental  form $\beta$ of $S_f$ is given by
$$
\beta =
\beta_{ij} dx^i dx^j =\frac{u_{ij}}{\sqrt{\det(g)}}dx^idx^j\,.
$$
Thus, the trace of the shape operator $A=g^{-1}\circ\beta$ of the hypersurface $S_f$ is given by
\begin{equation}\label{eqn:tr.A.euc}
\tr(A)=\tr(g^{-1}\circ\beta)=\sum_{i,j=1}^n\frac{\big(\det(g)\delta_{ij}-u_iu_j\big)u_{ij}}{\det(g)^{\frac32}}\, ,
\end{equation}
so that one can easily obtain the local expression of a $\SE(n+1)$--invariant PDE in view of Theorem \ref{th.euc.main}. For instance,
%
%
%
in the case of $n=2$, i.e., of $2$ independent variables,
$\tau_1  =  \tr(A)$  and the  equation $\tau_1=0$, in view of \eqref{eqn:tr.A.euc}, gives the  classical  equation  for minimal hypersurfaces:
$$
(1+u_2^2)u_{11}-2u_1u_2u_{12}+(1+u_1^2)u_{22}=0\,.
$$
Also, note that the classical Monge-Amp\`ere equation is obtained as follows:
$$
\frac12\bigg( \big(\tr (A)\big)^2-\tr(A^2)\bigg)=\frac{u_{11}u_{22}-u_{12}^2}{1+u_1^2+u_2^2}=0\,.
$$

\section{Invariant PDEs for hypersurfaces of   $\SSS^{n+1}$}\label{secCasConf}
In this section, we  describe  second order $G$--invariant PDEs    in  the  case  when   $M= G/H : = \SSS^{n+1} = \SO(1,n+2)/\mathsf{Sim}(\EE^{n+1})$   is  the  conformal sphere, that is the sphere $\SSS^{n+1}$ endowed with the conformal class $[g]$ of the standard metric $g$,  considered  as  a homogeneous manifold  of  the  conformal  group $G=\SO(1,n+2)$, called also the M\"obius (or Lorentz) group. We   use  the  standard  model of   the  conformal  sphere as the projectivisation of  the light  cone in  the Minkowski vector  space   $\R^{1,n+2}$.  The  stabilizer  of  the point  $o = \R p$, where $p$ is an isotropic  line, is isomorphic  to  the  group  $\mathsf{Sim}(\EE^{n+1})$ of  similarities  of  the  Euclidean   space $\EE^{n+1}$.

\subsection{Geometry  of the  conformal  sphere}

\subsubsection{The    standard  decomposition   of  the  Minkowski    space $ W=\R^{1,n+2} $  and of the  M\"obius  group $G=\SO(W)$}
  Let    $W=\R^{1,n+2}$  be   the  pseudo--Euclidean vector space  with   an orthonormal basis
\begin{equation}\label{eqn:iso.basis}
\{p,e_0,e_1,\ldots, e_{n}, q\}\, ,
\end{equation}
where $p$ and $q$ are isotropic vectors.
With respect to the  basis \eqref{eqn:iso.basis}, we have the decomposition
\begin{equation}\label{eqDecDabliu}
W=\R p \oplus E \oplus \R q = \R p\oplus (\R e_0 \oplus E^{e_0}) \oplus \R q   \, ,
\end{equation}
where
$$
E^0:=\langle e_1,\dots,e_n \rangle\,.
$$
We shall denote by $g_W$ the Minkowski metric on $W$.
\begin{remark}\label{eqRemarkPerProiezioni}
The  Euclidean   subspace  $E$   is   the   orthogonal  complement of  the   hyperbolic  plane  $\R p \oplus \R q$: therefore, it  is not  determined  by    the  isotropic  vector  $p$,   but  the  canonical  projection
$\pi: E \to  \overline{E} := p^{\perp}/\R p  $  is   an isometry  of $E$ onto the    factor space $\overline{E}$, equipped  with  the induced Euclidean metric.  We   denote  by
$$
\overline{E} = \R \overline{e}_0 \oplus \overline{E}^{e_0}
$$
the orthogonal  decomposition  of  $\overline E$: it is  the projection through $\pi$  of  the   orthogonal  decomposition
$
E = \R e_0 \oplus E^{e_0}\,.
$
\end{remark}
Decomposition \eqref{eqDecDabliu} can be regarded as a depth--one gradation of the linear space $W$, which  induces   the   following gradation of  the Lie algebra $\g=\so(W) = \bigwedge^2W $ of the  M\"obius   group  $G=\SO(W)$:
\begin{equation}\label{eqGradGi}
\g=\g^{-1}\oplus \g^{0}\oplus\g^{+1}= \underset{\deg=-1}{\underbrace{(\R q\wedge E)}} \oplus \underset{\deg=0}{\underbrace{(\R (p\wedge q)\oplus \so(E))}}\oplus \underset{\deg=+1}{\underbrace{(\R p\wedge E)}}\, .
\end{equation}
The Lie algebra gradation \eqref{eqGradGi} integrates to a (local, i.e.,  defined  in  some  open  dense  domain) decomposition  of  the M\"obius   group
\begin{equation}\label{eqGradGiGRUPPI}
\SO(W)\stackrel{\textrm{loc.}}{=}G^{-1}\cdot G^0\cdot G^{+1},
\end{equation}
where
\begin{eqnarray*}
G^0&=&\CO(E)\,, \\
 G^{-1}&=&\left\{ \left(\begin{array}{ccc}1 & 0 & 0 \\-\xi & \id & 0 \\-\frac{1}{2}\|\xi\|^2 & \xi^t & 1\end{array}\right) \mid \xi\in E\right\}\simeq E\, ,
\\
G^{+1}&=&\left\{ \left(\begin{array}{ccc}1 & \xi^t & -\frac{1}{2}\|\xi\|^2 \\  0 & \id &-\xi\\ 0 & 0 & 1\end{array}\right) \mid \xi\in E\right\}\simeq E\, .
\end{eqnarray*}
Note  that   the  groups   $G^{\pm 1 }$ are isomorphic  to  the vector  group $E = \R^{n+1}$  and  that  we always  consider   vectors as  column--matrices .
One can  check  directly  that
$\SO(E)\cdot G^{+1}$ (resp., $\SO(E)\cdot G^{-1}$) is the stabilizer $G_p$ (resp., $G_q$) of the point $p$ (resp., $q$) in $G$.

\subsubsection{The conformal sphere as projectivised  light cone   $\p W_0$ in $W$}
The  isotropic  cone
\begin{equation*}
W_0=\{ 0\neq w\in W\mid w^2 =0\}
\end{equation*}
(the set of non--zero isotropic vectors in $W$)    is  a   homogeneous  manifold  of  the  M\"obius  group   $G = \SO(W)$: $W_0=G/G_p$.
The \emph{conformal sphere} $\SSS^{n+1}$  is  defined  as the projectivization of the isotropic cone:
\begin{equation*}
\SSS^{n+1}:=\p W_0\, .
\end{equation*}
 The  M\"obius  group  $G$  acts   transitively  on $\SSS^{n+1}$. We  consider   the isotropic  line
    \begin{equation*}
o:=\ell_0:=[p]\, ,
\end{equation*}
where   $p$ is  as  in  \eqref{eqDecDabliu}. Then  $\SSS^{n+1}$  is identified  with  the  homogeneous   space
\begin{equation*}
\SSS^{n+1}=G/H=G/G_{[p]}\, ,
\end{equation*}
where $H=G_{[p]}=G_0\cdot G^{+1}$ is the stabilizer of    the origin $o =[p]$:
\begin{equation*}
G_{[p]} =\left\{
\left(
\begin{array}{ccc}
a & \eta^t & -\frac{a}{2} \|\eta\|^2
\\
0 & B & -a \eta
\\
0 & 0 & \frac{1}{a}
\end{array}
\right)
\,,\quad B\in \SO(E)\, ,\ \eta\in E\, ,\  a\in\R  \right\}.
\end{equation*}
In other words,
\begin{equation}\label{eqAccaCasConf}
H=G_{[p]} = \mathsf{Sim}(E)=G^{+1}\rtimes\CO(E)=E\rtimes\CO(E)
\end{equation}
is isomorphic   to  the \emph{group of similarities} $\mathsf{Sim}(E)$ of the Euclidean space $E$.\par
The Lie algebra $\gh=\g_{[p]}$ of $G_{[p]}$ is the  stability (parabolic) subalgebra
\begin{equation*}
\gh=\g_{[p]}=\g^{0}\oplus\g^{+1}=  (\R (p\wedge q)\oplus \so(E))\oplus (\R p\wedge E)\, ,
\end{equation*}
given by the non--negative part of \eqref{eqGradGi}.

\subsubsection{The isotropy   group $j(H)$,  the tangent  bundle $T\SSS^{n+1}$  and the  bundle   $J^1=J^1(n, \SSS^{n+1})  $  of  the  conformal  sphere}
The   tangent space to $\SSS^{n+1}$ at $o  = \R p$ is given by
\begin{equation}\label{eqFormulaToEsse}
T_o\SSS^{n+1}   = \overline{E}={p^\perp}/{\R p} \simeq \g/\g_{[p]}=\g^{-1} =  \R q \wedge  E\, .
\end{equation}
The isotropy representation
\begin{equation}\label{eqn:j.isotropic}
j : H=G^{+1}\rtimes\CO(E) \to  \GL(\overline{E})
\end{equation}
in the  tangent space \eqref{eqFormulaToEsse}  has  kernel    $G^{+1}$    and  it reduces  to the standard action of the linear  conformal   group  $j(H)= \CO(\overline E) = \R^+ \times\SO(\overline E)$ on $\overline{E} = \langle \overline{e}_0, \overline{e}_1, \dots, \overline{e}_n\rangle$.  We
have  the natural identification
$$
T\SSS^{n+1} =   G \times_{j(H)} \overline E \longrightarrow  \SSS^{n+1} = G/H
$$
of the  tangent bundle $T\SSS^{n+1} $   with  the homogeneous  vector  bundle $G \times_{j(H)} \overline E$  over
$G/H$,
defined  by  the isotropy  action \eqref{eqn:j.isotropic}.\par
Denote  by    $[g_o]$    the  $\CO(E)$--invariant  conformal metric  on  the tangent  space  \eqref{eqFormulaToEsse};  it defines a $G$--invariant  conformal metric  $g$ on  the    sphere $\SSS^{n+1}$.\par
%
%
%
%
To  simplify notation,  we  set
$$
V = \overline{E}^{e_0}\, ,
$$
so that it is possible to identify $V$   with $ {E}^{e_0}$ (see   above Remark \ref{eqRemarkPerProiezioni}).
The  hyperplane   $V \subset \overline{E} = T_o \SSS^{n+1}$  is  a point of  the  space   $J^1 = J^1(n, \SSS^{n+1})$. We shall  denote such point also by $o^1$.

Recall  that    the  fiber   $J^1_o$   of  the  bundle  $J^1$  at  the point  $o = [p]$ is  identified  with the    Grassmannian  $\Gr_n(\overline E)$ of  hyperplanes of $\overline{E}=T_o\SSS^{n+1}$ and then with the  projective  space $\p\overline{E}^*$. The isotropy  group  $j(H)$ acts transitively  on  this  space  and  the  stability     subgroup
$j(H)_{o^1}$ of  $o^1  =  V $ is   $\CO(V) = \R^+ \times \OOO(V)$.
Thus, we  get the following proposition.
\begin{proposition}\label{prop.dmitri.H1.G1.COV}
  The  M\"obius   group   $G$  acts  transitively  on  $J^1 = J^1(n, \SSS^{n+1})$   with  the  stabilizer of  the point  $o^1 =  V$ given by
  $H^{(1)}  = G^{+1} \rtimes   \CO(V)$. In particular,   $J^1 = G/ H^{(1)}$.
\end{proposition}

\begin{corollary}
In terms of   Lie  algebras,  the isotropy  action  of   $\gh = \g^0 \oplus \g^{+1}$  on $T_o  \SSS^{n+1}= \overline E = \R \overline{e}_0 \oplus V$
satisfies
$$
\ker j =  \g^{+1} = \R p \wedge  E\,,\,\,  j(\gh) = j(\g^0) = \mathfrak{co}(\overline E)\,.
$$
The   stability   subalgebra  of  the point  $o^1 = V$  in  $\gh$   is
\begin{equation}\label{eqn:h1.useful}
\gh^{(1) } =   \R p \wedge E \oplus \R p\wedge q  \oplus \mathfrak{so}(V)
= (\so(V)\oplus V)\oplus(\R p \wedge q  \oplus \R e_0\wedge p )
\end{equation}
and,   moreover,
\begin{equation*}
j(\gh^{(1)}) = j(p \wedge q) \oplus j(\mathfrak{so}(V)) \, ,
 \end{equation*}
 where
\begin{equation*}
j(p\wedge q)=-\id \end{equation*}
 and
\begin{equation*} j(\mathfrak{so}(V))e_0 =0,\,\,  j\big(\mathfrak{so}(E^{e_0})\big)|_V =  \mathfrak{so}(V).
 \end{equation*}
\end{corollary}

\subsection{Standard   coordinates   of   the  conformal  sphere $\SSS^{n+1}$}

To  describe   the   fiber  $J^2_{o^{1}}$,   we  define    an  appropriate  coordinate  system in  $\SSS^{n+1}$.
Let us consider the system of coordinates
\begin{equation}\label{eqn.coord.che.mi.servono}
(\lambda,u,\x,s):=(\lambda,u,x^1,\dots,x^n, s)\,  
\end{equation}
in  $W$ associated  to  the  basis \eqref{eqn:iso.basis},
   such  that  $W \ni w = \lambda p + u e_0 + \sum x^i e_i + s q $. We set
$$
x=\sum_{i=1}^n x^i e_i\,,\quad \|x\|^2=\sum_{i=1}^n(x^i)^2\, .
$$
Coordinates \eqref{eqn.coord.che.mi.servono}  are  homogeneous  coordinates  of  the projective  space   $\p W$.  Taking  $\lambda =1$,   we  consider  $(u,x^i, s )$   as  associated  local  affine  coordinates  in $\p W$.   Then  the   conformal    sphere  $ \SSS^{n+1}= \p W_0   $  has local  coordinates  $(u, x^i)$  such  that
$$
w = p + u e_0 + x  + s(x) q \in \SSS^{n+1}\, ,
$$  where
\begin{equation}\label{eqn:sx}
s(u,x) = - \frac12 (u^2+\|x\|^2) \,.
\end{equation}
          We call   such  coordinates   the {\it  standard   coordinates   of   the  conformal  sphere}.   They  depend  on
          an isotropic vector  $p \in \R p =o$,  on the lift   $E\subset W$   of  the   tangent   space  $\overline{E} =  p^{\perp }/\R p $  and on an orthogonal  decomposition $E = \R e_0 \oplus E^{e_0}$.   Then  the   isotropic  vector
            $q$ is   defined    as  the  vector    $q \in  E^{\perp}$  with  $p \cdot  q =1$.

\subsection{Hyperspheres in  $\SSS^{n+1}$    with  fixed $1$--jet   $o^{1}   = V$}
Below we prove  that   the  set of the  hyperspheres $S$ of  $\SSS^{n+1}$
through   the point $o$ and with  given  tangent   space   $T_oS = o^{1} =V$ forms a $1$--parametric  family which is  an orbit of  the  stability  group $H^{(1)}$ of  the point $o^{1}$.  To see this we  calculate the  second  jet   $[S]^2_o$ of   a hypersphere $S$ in   local  coordinates.


 Let $e \in W$ be  a unit  spacelike   vector,  i.e.,  $e\cdot e =1$,   and  let $W^{e} = e^{\perp}$ be  the orthogonal  hyperplane to $e$.
\begin{definition}
The   projectivization   $S^e = \p W^{e}_0$  of  the  isotropic   cone   $W^e_0$ is  a hypersurface  of  the  conformal  sphere   $\SSS^{n+1} = \p W_0$,  which is   called   a  \emph{hypersphere}.
\end{definition}
The   vector  $e$  defines   an orientation  of    $S^e$. Two hyperspheres (respectively, oriented  hyperspheres)  $S^e, \, S^{e'}$  coincides  if  and  only   $e$ and  $e'$  coincides up  to  sign (respectively,  $e = e'$).
   Hence,   the  set   of oriented  hyperspheres   is identified  with  the
   anti  de Sitter  space   $W_1 = \{ e \in  W, \,  e \cdot e =1\}$  of  unit  vectors,   which is    the  homogeneous   space    $W_1   =  G/ G_e   = \SO(1,n+2)/\SO(1,n+1)$  of     the M\"obius   group  $G$.
    An  element   $g \in  G$   acts  on   a hypersphere   $S^e$  by
       $$    g S^e  =  S^{ge},\,   \, g \in  G.$$

\begin{remark}
Let us consider the basis \eqref{eqn:iso.basis}.  Let
$$
W =  \R p \oplus E \oplus \R q  =  \R p \oplus  (\R e_0 \oplus  E^{e_0})  \oplus \R q
$$
be the  associated  decomposition, and recall  that
   $$T_o \SSS^{n+1}  = \overline E := p^{\perp}/\R p =  (\R p  \oplus  E)/ \R  p \simeq E\, .$$
  Note  that   $S^{e_0}$ is an oriented  hypersphere  through  the point  $o = [p]$  with   the tangent   space
  $$ T_oS^{e_0} =o^{1}:= V  = (\R p \oplus   E^{e_0})/\R p\, . $$
\end{remark}
  \begin{lemma}  The   set  of  the oriented  hyperspheres $S^e$ through the  point $o$  with a given tangent  space  $ T_o S^e = V$   forms   the    $1$--parametric  family  $S^{e_0 - t p},\,  t \in  \R$,  which is   an orbit of  the  action   $\tau$  of  the stability  subgroup  $H^{(1)} =  G^{+1}\rtimes \CO(\overline{E}^{e_0})$  of  the point  $o^{1}$ on  the  fiber   $J^2_{o^{1}}$. More  precisely,  the $1$--parametric   subgroup
$$
A_t^{e_0}:= \exp t(e_0 \wedge p)
$$
of $G^{+1} = \ker j$
acts   transitively on  the  set of the  hyperspheres of  $\SSS^{n+1}$
through  the point $o$ and with fixed tangent   space   $o^{1} =V$
and  transforms  $S^{e_0}$ into   $S^{e_0 -  t p}$.
\end{lemma}
\begin{proof}
Let  $S^e$ be  an  oriented  hypersphere  through  the point  $o= [p]$.  Then  $T_o S^e $ is  a hyperplane in
$T_o \SSS^{n+1}  = \overline{E} = (\R p \oplus \R e_0 \oplus E^{e_0})/\R p$, orthogonal  to  $e$.  If   $T_o S^e = V=(\R p \oplus   E^{e_0})/\R p$,   then  the  unit  vector  $e =  \pm e_0  + \mu p$, for some $\mu\in\R$. If the hyperspheres  $S^e$ and $S^{e_0}$  have  the  same orientation,  then
  $e =  e_0 + \mu p$. Since   the  group $G$ acts transitively both on the set of  hyperspheres (which  is isomorphic to the anti de Sitter space of unit space--like vectors) and on $J^1$,  the  stability group   $H^{(1)}$ of the point $o^{1}$ acts transitively  on  the  set  of hyperspheres  with  fixed $1$--jet  $o^{1}$.
We describe  the action of  the  $1$--parametric subgroup $A^{e_0}_t = \exp t(e_0 \wedge p)$ on $S^{e_0}$.
 Since $e_0 \wedge p $  acts by  $q \to e_0 \to  -p \to 0 \,, E^{e_0} \to 0$,  we  get
     $A^{e_0}_t = \id +  te_0 \wedge p + \frac12 t^2 (e_0 \wedge p)^2 $.  Thus, we obtain the  following  formula  for  the  action of  $A^{e_0}_t$:
%
\begin{equation}\label{actionofG^1}
A^{e_0}_t:
\left\{
\begin{array}{l}
 p\to p
\\
 e_0\to e_0 - t p
\\ x\to x
\\
q\to q + t  e_0 - \frac12  t^2 p
\end{array}
\right.
\end{equation}
This  shows  that   $A^{e_0}_t (S^{e_0}) = S^{e_0 - t p}$.
\end{proof}

\subsubsection{The affine action  $\tau$ of  the   stability  subgroup $H^{(1)}$  on  the  fiber  $J^2_{o^1}$ and  $\SO(1,n+2)$--invariant PDEs}\label{action(tau)}

In view of \eqref{eqn:h1.useful}, we write the group $H^{(1)}$ as the direct product of the conformal group $V\rtimes \CO(V)$ of $V$ and a $1$--dimensional central subgroup $A^{e_0}$ generated by the Lie algebra $\R e_0\wedge p$:
\begin{equation}\label{eqn:H1.new.conf}
H^{(1)}=(V\rtimes \CO(V))\times A^{e_0}\, ,
\end{equation}
where
$$
A^{e_0}=\exp\R e_0\wedge p =\{A^{e_0}_t\,,\,\,t\in\R\}\, .
$$
We will see that the conformal group $V\rtimes \CO(V)$ acts in the natural way on the space $S^2V^*$ as a linear group while  the central subgroup $A^{e_0}$ acts via  parallel translations in the direction of  $g\in S^2V^*$, where $g$ the restriction of the Minkowski metric $g_W$ to $E^{e_0} = V$:
\begin{equation}\label{eqn:g.W.V}
g:=(g_W)\big|_V\, .
\end{equation}
In order to show that $S^{e_0}$ is a fiducial hypersurface, we set
$$
o^{\ell}  := [S^{e_0}]^{\ell}_o\,
$$
and we  compute   the   $2$--jet $[S^{e_0 + \mu p}]^2_o $ at $o$ of   the  hypersphere $S^{e_0 + \mu p}$.
%
%
%
The   action  of  each  element   of   the stability  subgroup
 $H^{(1)} =  G^{+1} \rtimes \CO(V) \subset \SO(W)$  on    a hypersurface
\begin{equation}\label{eqn:Sf.conf.Dmitri}
S_f = \{ w=p +   f(x)e_0 +x+s(x)q \,,\,\,x\in E^{e_0}\}\,,\quad [S_f]^1_o =o^{1}\,,
\end{equation}
where $s(x):=s(f(x),x)$ is given by \eqref{eqn:sx},    can  be easily  described.  Actually, we only  need  to know the action of  the  elements $B \in \CO(V)$   and   $A^{e_0}_t \in  G^{+1} = \ker  j$.
%
%
%
Since   each  element  $B  \in  \CO(V) $  is  a linear transformation  which   acts  trivially on  $p$ and $q$,   we   have the following Lemma.
\begin{lemma}\label{lemma.che.ne.so}
Let   $S_f$ as in \eqref{eqn:Sf.conf.Dmitri} be a hypersurface such that $f(0)=0$,  $u_i =\tfrac{\partial f}{\partial x^i}(0)= 0$, so that $[S_f]^1_o = o^{1}=V$. 
Let
\begin{equation}\label{eqn:beta}
[S_f]^2_o  = u_{ij} dx^idx^j=: \beta \in S^2T_o^*S_f=S^2V^*\,, \quad u_{ij} = \tfrac{\partial^2 f }{\partial x^i\partial x^j}(0)\,.
\end{equation}
Then an  element  $B \in  \CO(V)$ transforms  the  hypersurface  $ S_f$  into     $B(S_f) =  S_{B^* f}$,  where  $(B^*f)(x) = f(B(x))$.
In particular, $B$ acts  on   the $2$--jet  $\beta$ in  the  standard  way
$$
(B^* \beta)(x,x) = \beta(B(x),B(x))\,,\,\,x\in V\,.
$$
\end{lemma}
%


\begin{lemma}
The element $A_t^v:=\exp{t v\wedge p}$
acts  trivially   on  the  fiber $J^2_{o^1} = S^2 V^*$, for any $v \in  V$.
\end{lemma}
\begin{proof}
It is a straightforward computation based on the formula
$$
A_t^v = \mathrm{id} + t v\wedge p - \frac12  t^2 \|v\|^2 p \otimes p\, ,
$$
where $p\otimes p$ is meant as an endomorphism via the metric \eqref{eqn:g.W.V}.

\end{proof}

In view of \eqref{actionofG^1} and \eqref{eqn:Sf.conf.Dmitri}, we have the  following   description  of  the  action  of  $A^{e_0}_t$  on  the  hypersurface   $S_f$ with    parametric  equation   $u = f(x)$.
\begin{lemma}
In the hypotheses of Lemma \ref{lemma.che.ne.so}, there exists a $\widetilde{f}_t$ such that
$$
A^{e_0}_t (S_f) = S_{\widetilde{f}_t}  = \left\{
\left(1-f(x)t-\frac12 s(x) t^2\right)p+  (f(x)+ts(x))e_0 + x +  s(x)q  \right\}\, .
$$
In particular,  the  action  of $A^{e_0}_t$  on $[S_f]^2_o  = \beta$ is
$$
A^{e_0}_t ([S_f]^2_o )  = [S_{\widetilde{f}_t}]^2_o  =  [S_f]^2_o  - \frac{1}{2}t  g
= \beta - \frac{1}{2}t  g\,,
$$
where $g$ is given by \eqref{eqn:g.W.V}.
\end{lemma}
\begin{proof}
It is a straightforward computation based on the construction of the desired $\widetilde{f}_t$:
$$
\widetilde{f}_t(\widetilde{x})= \frac{f(x)+ts(x)}{1-f(x)t-\frac12 s(x) t^2}\,,\quad \widetilde{x}=\frac{x}{1-f(x)t-\frac12 s(x) t^2}\,.
$$
\end{proof}
In view of the above lemmas and recalling that $J^2_{o^1}=S^2V^*$, we see that   the action $\tau:H^{(1)}\to \Aff(S^2V^*)$ is given by:
\begin{eqnarray*}
\tau(V)&=&\mathrm{Id}\,,\\
 \tau(B)(\beta)&=&\beta(B(\cdot),B(\cdot))\, ,\quad B\in\CO(V)\,, \\
\tau(A^{e_0}_t)(\beta)&=&\beta-\frac12 t g\,.
\end{eqnarray*}
The next corollary follows from the fact that  the hypersphere $S^{e_0} = \{p + x + s(x) q\}$ is  defined  by the    equation   $u(x) = 0 $.
\begin{corollary}
The   $1$--parametric family  $S_t := S^{e_0 - t p} =   A^{e_0}_t (S^{e_0})$ of  hyperspheres of $\SSS^{n+1}$  with  tangent  space $o^{1} =V$ has    $2$--jets   $[S_t]^2_o = -\frac12 t g$.
Hence,  the orbit  $H^{(1)}\cdot o^{2} = \{A^{e_0}_t(o^2)=-\frac12 t g \,,  t \in \R \}$
and the hypersphere    $ S^{e_0}$ is  a  fiducial hypersurface.
\end{corollary}

\begin{corollary}
The central subgroup $A^{e_0}$ of the group $H^{(1)}$ (see \eqref{eqn:H1.new.conf}) acts on the fiber $J^2_{o^1}=S^2V^*$ as parallel translation along the line $\R g$ and the normal subgroup $V\rtimes\CO(V)$ acts by linear transformation $\ker(\tau)=V$ in a natural way. In particular the conformal sphere $\SSS^{n+1}$ is a $2$--admissible homogeneous manifold.
\end{corollary}

\subsection{Construction of $\SO(1,n+2)$--invariant PDEs}
 Now  we  are  ready  to   give  a  construction  of  all  $\SO(1,n+2)$--invariant second order  PDEs  for hypersurfaces in  $\SSS^{n+1}$.

Like in the Euclidean   case,  Theorem \ref{thMAIN1}  reduces  the description  of  such  PDEs
  to  the description of  $\CO(V)$--invariant     hypersurfaces    $ \overline{\Sigma} \subset S^2_0(V^*)$, where $S^2_0(V^*)$ is the space of  trace--free  quadratic  forms on  the tangent  space $V = T_oS^{e_0} =o^{1}$.
 According  to our  strategy,  we  have  to  construct   a  quotient  of the  affine  space   $J^2_{o^{1}}$  where the action  of  the  group   $H^{(1)}$  becomes linear.   Recall  that the  second  jet $[S_f]^2_{o}$ of  a hypersurface $S_f$  with $[S_f]^1_o =o^{1}$    is   represented   by   the    quadratic  form   $\beta$, see \eqref{eqn:beta}.
\begin{definition}
The traceless part $A_{\circ}$ of  the  shape operator $A=g^{-1}\circ\beta$, where $g$ is as in \eqref{eqn:g.W.V} and $\beta$ as in \eqref{eqn:beta} is called the \emph{conformal  shape operator  of $S_f$} at  the point $o$.
\end{definition}
Let us observe that   the conformal shape  operator  ${A_{\circ}}_{a}$ of  a hypersurface $S_f$    of $M$  is   well  defined  at  any  point $a\in S_f$ and it  
depends only upon  $[S_f]^2_a$. Moreover,
the  action of  the   group $G_{a^1}$  on ${A_{\circ}}_{a}$ reduces  to  the standard action of the group $\CO(V)$ on the space of traceless forms $S_0^2V^*$ (see also \cite{MR3603758}).  Finally, let us recall that the relative invariants of $S_0^2V^*$ with respect to the group $\CO(V)$ are homogeneous  polynomials
$$
F=F(\sigma^{\circ}_2,\dots,\sigma^{\circ}_n)\,,\,\,\deg(\sigma^{\circ}_h)=h\,,
$$
where the $\sigma^{\circ}_i=\sigma^{\circ}_i(k_1,\dots,k_n)$'s are the elementary symmetric functions of the eigenvalues of the conformal shape operator $A_{\circ}$, or, equivalently,
$$
F=F(\tau^{\circ}_2,\dots,\tau^{\circ}_n)\,, \,\,\tau^{\circ}_h:=\tr(A_{\circ}^h)\,.
$$
Such an invariant polynomial
defines the $\CO(V)$--invariant algebraic hypersurface
\begin{equation}\label{eqn:euc.F.}
\overline{\Sigma}=\{F=0\}\subset S^2_0V^*\,.
\end{equation}
The associated hypersurface $\E=\E_{\overline{\Sigma}}=\SO(1,n+2)\cdot {\Sigma}\subset J^2$ is an $\SO(1,n+2)$--invariant hypersurface, that defines a second order PDEs polynomial in the second derivatives.
Solutions to PDE $\E_{\overline{\Sigma}}$ are hypersurface $S\subset M$ whose conformal shape operator ${A_{\circ}}_{\pp}$ satisfies, $\forall\,\pp\in S$, the equation
\begin{equation}\label{eqn:conf.final.form.F}
F({\tau^{\circ}_2}_{\pp},\dots,{\tau^{\circ}_n}_{\pp})=F\left(\tr({A_{\circ}^2}_{\pp}),\dots,\tr({A_{\circ}^n}_{\pp})\right)=0\,.
\end{equation}

%
%
%
%
In analogy with the Euclidean case, we can once again conclude that   Theorem \ref{thMAIN1} implies the following one.
\begin{theorem}\label{th.conf.main}
Any second order $\SO(1,n+2)$--invariant PDE for  hypersurfaces   of  $\SSS^{n+1}$ which is a polynomial in the second--order derivatives
is of the form \eqref{eqn:conf.final.form.F}, where $F$ is a homogeneous polynomial, $\deg(\tau^{\circ}_h)=h$.
\end{theorem}

\subsubsection{Description  of  $\SO(1,n+2)$--invariant  PDEs in local  coordinates}\label{sec:conf.local}

In the present section one can use the local description of $g^{-1}$ and $\beta$ contained in Section \ref{sec.euc.local}

\smallskip

Let $A=g^{-1}\circ\beta$ the shape operator of a hypersurface $S_f=\{u=f(\x)\}$, where $g$ and $\beta$ are, respectively, its first and the second fundamental form.
The traceless second fundamental form $\beta_{\circ}$ and the conformal shape operator of a hypersurface $S_f=\{u=f(\x)\}$ is
$$
\beta_{\circ}=\beta-\frac{1}{n}\tr(A)g = \beta-Hg\,, \quad A_{\circ}=g^{-1}\circ\beta_{\circ}=A-\frac{\tr(A)}{n}\mathrm{Id}=A-H\,\mathrm{Id}\, ,
$$
where $H$ is the mean curvature.

\smallskip
For $n=2$ the only relative conformal invariant is
$$
\det(A_{\circ})\,\left(=-\frac{1}{2}\tr(A_{\circ}^2)\right)=\det(A)-\frac{1}{4}\tr(A)^2=K-H^2\, ,
$$
where $K$ is the Gaussian curvature. We underline that the quantity $H^2-K$ is the coefficient of the Fubini first conformally invariant fundamental form.\footnote{In the work  \cite{FubiniApplicabilita} three of us (JG, GM and GM) have clarified the role of the conformally invariant fundamental form in the theory of PDEs of Monge--Amp\`ere type.} Also, $H^2-K=\frac14(k_1-k_2)^2$, $k_i$ being the principal curvatures, so that $\E:=\{H^2-K=0\}$ describes points having the same principal curvatures, i.e., umbilical ones. A deeper analysis shows that $\E$ is a \emph{system} of two PDEs. Indeed,  examining $\E$
over the point $o^{1}$, i.e., with $u_1 = u_2 = 0$, one   obtains
a sum of squares. To see this it is enough to recall that
\begin{eqnarray*}
H&=&\frac{1}{2}\frac{(u_2^2+1)u_{11}-2u_1u_2+(u_1^2+1)u_{22}}{(u_1^2+u_2^2+1)^{\tfrac{3}{2}}}\, ,\\
K&= &\frac{ u_{11} u_{11}-u_{12}^2}{(u_1^2+u_2^2+1)^{2}}\, ,
\end{eqnarray*}
and then to  replace the values $u_1=0$, $u_2=0$  in
\begin{equation*}
H^2-K= \frac{1}{4}\frac{((u_2^2+1)u_{11}-2u_1u_2+(u_1^2+1)u_{22})^2}{(u_1^2+u_2^2+1)^{3}}-\frac{ u_{11} u_{22}-u_{12}^2}{(u_1^2+u_2^2+1)^{2}}\, ,
\end{equation*}
which yelds immediately
\begin{equation*}
(H^2-K)_{o^{1}}=\left( \frac{1}{2}u_{11} -\frac{1}{2}u_{22}   \right)^2+u_{12}^2\, .
\end{equation*}
By invariance, it follows that  also the whole  PDE $\E$ is a   subset of codimension $2$.\par

For $n=3$, taking also into account that the characteristic polynomial of $A_{\circ}$ is $\sum_{i=0}^n\sigma^{\circ}_{n-i}\lambda^i$ and that of $A$ is $\sum_{i=0}^n\sigma_{n-i}\lambda^i$,  we have that
\begin{eqnarray*}
\sigma^{\circ}_3&&=\det(A_{\circ})=\frac{1}{3}\tr(A_{\circ}^3)=\frac{2}{27}(\tr(A))^3 +\frac{1}{3}\tr(A)\sigma_2+ \det(A)=\\ &&=2H^3+H\sigma_2+K\, ,\\
\sigma^{\circ}_2&&=\frac{1}{2}\tr(A_{\circ}^2)=\frac{1}{3}(\tr(A))^2 +\sigma_2 = 3H^2 +\sigma_2\, .
\end{eqnarray*}
Thus, for instance,
$$
2H^3+H\sigma_2+K=0\,,\quad 3H^2 +\sigma_2=0\,,\quad (2H^3+H\sigma_2+K)^2 +  (3H^2 +\sigma_2)^3=0
$$
are $\SO(1,5)$--invariant PDEs.

\subsection*{Acknowledgements} Dmitri Alekseevsky   has been partially  supported  by grant no. 18-00496S of the Czech Science Foundation.
Giovanni Moreno has been   partially founded by the Polish National Science Centre grant under the contract number 2016/22/M/ST1/00542. Gianni Manno was partially supported by the ``Starting Grant per Giovani Ricercatori'' 53\_RSG16MANGIO of the Polytechnic of Turin. This work was also partially supported by the following projects and grants:
\begin{itemize}
\item ``Connessioni proiettive, equazioni di Monge-Amp\`ere e sistemi integrabili'' by Istituto Nazionale di Alta Matematica (INdAM);
\item ``MIUR grant Dipartimenti di
Eccellenza 2018-2022 (E11G18000350001)'';
\item ``Finanziamento alla ricerca 2017-2018 (53\_RBA17MANGIO)'';
\item PRIN project 2017 ``Real and Complex Manifolds: Topology, Geometry and holomorphic dynamics'';
\item Grant 346300 for IMPAN from the Simons Foundation, matching 2015--2019 Polish MNiSW fund.
\end{itemize}
Gianni Manno and Giovanni Moreno are members of GNSAGA of INdAM.

\newpage

\bibliographystyle{abbrv}
\bibliography{../../../../0)_MIE_RICERCHE/BibUniVer/BibUniver}


\end{document}